\documentclass{siamart220329}

\usepackage{diagbox}
\usepackage{multirow}
\usepackage{cite}
\usepackage{url}
\usepackage{amsfonts,amssymb} 
\usepackage{amsmath} 
\usepackage{color} 
\usepackage{graphicx} 
\usepackage{epsfig,mathrsfs} 
\usepackage{graphicx}

\usepackage{caption} 
\usepackage{color}
\usepackage{cases}
\usepackage{subfigure}
\graphicspath{{figure/}}
\usepackage{setspace}
\usepackage{dsfont}
 \usepackage{tabularray}
  \usepackage{array}
 \usepackage{multirow}

\def\Re {{\mathbb R}}

\newtheorem{remark}{Remark}[section] 
\newtheorem{example}{Example}[section] 

\usepackage{booktabs}
\usepackage{mathtools}

\title{Error analysis of a collocation method on graded meshes for  nonlocal  diffusion problems with weakly singular kernels
	\thanks{Submitted to the editors DATE.
		\funding{The work was supported by the Science Fund for Distinguished Young Scholars of Gansu Province under Grant No. 23JRRA1020, and National Science Fund for Distinguished Young Scholars (11925204).}}}

\author{ Minghua Chen  \thanks{School of Mathematics and Statistics, Gansu Key Laboratory of Applied Mathematics and Complex Systems,
		Lanzhou University, Lanzhou 730000,  China.
		Email address: chenmh@lzu.edu.cn}
	\and Chao Min \thanks{School of Mathematics and Statistics,  Lanzhou University, Lanzhou 730000,  China.}
	\and  Jiankang Shi  \thanks{School of Mathematics and Statistics,
		Lanzhou University, Lanzhou 730000,  China.}
    \and  Jizeng Wang \thanks{
Key Laboratory of Mechanics on Disaster and Environment in Western China, the Ministry of Education,
College of Civil Engineering and Mechanics, Lanzhou University, Lanzhou 730000, China
Email address: jzwang@lzu.edu.cn}
}

\begin{document}
	
	\maketitle

	\begin{abstract}
Can graded meshes yield   more accurate numerical solution than uniform meshes?
A time-dependent nonlocal diffusion problem with a weakly singular kernel is considered using  collocation method.
For its steady-state counterpart, under the sufficiently smooth solution, 
we first clarify that the standard graded meshes are worse than  uniform meshes and may even lead to divergence; instead,  an optimal convergence rate  arises in so-called anomalous graded meshes. 
Furthermore, under  low regularity solutions, it  may suffer from a severe order reduction in (Chen, Qi, Shi and Wu, IMA J. Numer. Anal., 41 (2021) 3145--3174).  
In this case,  conversely, a sharp error estimates appears in standard graded meshes, but offering far less than first-order accuracy. 
For the time-dependent case, however, second-order convergence can be achieved on graded meshes. 
The related analysis are easily extended  for  certain multidimensional problems.
Numerical results are provided  that confirm the sharpness of the error estimates.
	\end{abstract}
	
	
	
	\thispagestyle{plain}
	\markboth{M. CHEN, C. MIN, J. SHI AND J. WANG}{COLLOCATION METHOD  FOR NONLOCAL PROBLEMS}

	\section{Introduction}\label{Se:intro}
 Nonlocal diffusion problems have been used to model very different applied situations  \cite{Andreu:10,SGP:13,Du:19}, for example in biology, mechanics, coagulation models, image processing,  particle systems,
and nonlocal anisotropic models for phase transition, etc.
  In this paper, 
we study an error estimates of a  collocation method on graded meshes 
for solving  the  time-dependent nonlocal diffusion problems  with   weakly singular kernels,  whose prototype equations are  \cite{Andreu:10,SGP:13}
\begin{equation}\label{equ1.1}
	\left\{ \begin{split}
		u_t(x,t) - \mathcal{L}u(x,t) &=f(x,t),    &  &  x \in \Omega,\, t>0,\\
		u(x,0)&=u_0(x),    & &   x \in \Omega,\\
		u(x,t)&=0,         & &   x \in \mathbb{R}\setminus \Omega,
	\end{split}
	\right.
\end{equation}
and its steady-state counterpart
\begin{equation}\label{equ1.2}
	\left\{ \begin{split}
		- \mathcal{L} u(x)      &=f(x), &  & x \in \Omega,\\
		u(x)&=0,        &    &x \in \mathbb{R}\setminus \Omega.
	\end{split}
	\right.
\end{equation}
The  nonlocal operator $\mathcal{L}$  is defined by, for $0<\alpha<1$,
\begin{equation}\label{equ1.3}
		\mathcal{L}u(x)=\int_{\Omega} \frac{u(y)-u(x)}{|x-y|^\alpha}dy ~~\forall x \in \Omega.
\end{equation}
 In this model, diffusion takes place inside $\Omega$ and $u$ vanishes outside $\Omega$. In the biological interpretation, there is a hostile environment outside $\Omega$, and any individual that jumps outside dies instantaneously \cite[p.\,41]{Andreu:10}.
The well-posedness (existence and uniqueness)  is addressed  in the monograph \cite[Chapter 3]{Andreu:10}.

Note that   Fredholm weakly singular integral equations of the second or third kind have the following form \cite{Atk:09,Zemyan:12}
\[
	\lambda(x) u(x)- \int^b_a \frac{u(y)}{|x-y|^\alpha}dy  =f(x),    \quad x \in (a,b) ,\quad 0< \alpha <1,
\]
where $\lambda(x)$  is a nonzero complex constant or vanishes at least once in the interval $[a,b]$, respectively.
Nevertheless, the nonlocal models  \eqref{equ1.2} may exhibit similarities but not    belong to the aforementioned Fredholm integral equations,
since the variable coefficients of  \eqref{equ1.2}  are great than zero, i.e., 
\[\lambda(x):=\int^b_a \frac{1}{|x-y|^\alpha}dy=\frac{1}{1-\alpha}\left[ \left(x-a\right)^{1-\alpha}+\left(b-x\right)^{1-\alpha} \right]>0,\quad 0< \alpha <1.\]
Indeed, there still exist some significant differences between these two models. 
For example, the inverse operators of Fredholm integral equations of the second kind  are uniformly bounded, see Theorem 4.2.1  in \cite{Atk:09}. However, the inverse operators of nonlocal models \eqref{equ1.2} are unbounded, as shown in Lemma 3.5 of \cite{Cao:20}.

There has  been some important progress in numerically solving nonlocal diffusion problems including  finite element methods \cite{WT:12}, collocation methods  \cite{TWW:13,ZhangGJ:16,CQSW:21}, fast conjugate gradient  method \cite{WT:12,Cao:20} and multigrid method \cite{CES:19}.
Among various techniques for solving nonlocal  problems \eqref{equ1.2}, collocation methods are the simplest,  but unfortunately, it poses more challenges in numerical analysis.
For example, Fredholm weakly singular integral equations of the  second kind have second-order convergence \cite[p.\,59]{Atk:09}. However, 
the numerical results for \eqref{equ1.2} with $\alpha=1$ indicate that the convergence rate appears to be close to 1.5 \cite{TWW:13}, although it remains to be proved.
In fact, it exhibits only first-order convergence  under  the   sufficiently smooth solution \cite{CQSW:21}.

As mentioned above in \cite{TWW:13,ZhangGJ:16,CQSW:21}, the convergence rates  both assume that the unknown solution has several continuous derivatives. However, the solution is rarely smoothly differentiable; instead, it often exhibits somewhat singular behavior in the neighborhood of the boundary. These phenomena arise naturally in the integral equation with singular kernels,
such as  Volettra integral equation\cite[p.\,340]{B:04}, Fredholm integral equation \cite{S:79}, time-fractional diffusion equation \cite{SOG:17}, 
and nonlocal/regional  fractional Laplacian  \cite{RS:14,CDMSM:23,HW:22,F:22} with the low regularity of the solutions  
\begin{equation}\label{equ1.4}
	\begin{split}
		\left|\frac{\partial^{\ell} }{\partial x ^{\ell}} u(x)\right|  
		& \le C[(x-a)(b-x)]^{\sigma-\ell}, ~0<\sigma<1 \ \text{ for }  \ell=0,1,2. 
	\end{split}
\end{equation}
To capture the multi-scale/singularities behaviors and achieve more accurate, the graded meshes technique 
are effectively employed  for such problems \cite{SOG:17,RS:14,HW:22}.

Can graded meshes yield  more accurate numerical solutions than uniform meshes?  For steady-state counterpart \eqref{equ1.2}, under the sufficiently smooth of the solutions, we  provide a  sharp error estimate in Theorem \ref{nnthm3.5}, namely, 
 \begin{equation*}
    \max_{1\leq i \leq 2N-1} |u(x_i)-u_{i}| 
\leq
\begin{cases}
  C N^{r-2}, & {\rm if}~~r>\frac{2}{3}, \\
  C N^{r-2}\ln N, & {\rm if}~~r=\frac{2}{3}, \\
  C N^{-2r},      & {\rm if}~~r<\frac{2}{3}.
\end{cases}
  \end{equation*}
Here  $u_i $ is the approximate solution of $u(x_{i})$,  $N$ is the number of grid points, and  $r> 1$ is the grading exponent of the  standard graded meshes; instead, for $0<r< 1$, is called anomalous graded meshes. 
These results imply   that the standard graded meshes are worse than the uniform grid and  even lead to divergence  for $r\geq 2$.
It also indicates that an optimal convergence rate (great than one) arises in anomalous graded meshes.  Conversely, a sharp error estimate (far less than one) appears in standard graded meshes with the low regularity of solution \eqref{equ1.4}, 
 since it possess the following  convergence rate  as shown in Theorem \ref{nnthm4.5}
\begin{equation*}
 \max_{1\leq i \leq 2N-1} |u(x_i)-u_{i}| 
\leq     \begin{cases}
C  N^{r-2} ,                        &{\rm if}~~  r> \frac{2}{\sigma+1},\\
C  N^{r-2} \ln N,                   &{\rm if}~~  r =\frac{2}{\sigma+1}, \\
C  N^{-r \sigma},         &{\rm if}~~  r <\frac{2}{\sigma+1}. \\
	\end{cases}
\end{equation*}
Finally, second-order convergence $\mathcal{O}\left( N^{-\min \left\{r \left(1+\sigma-\alpha\right), \,2 \right\} } \right)$ can be achieved on graded meshes
for the time-dependent nonlocal problems \eqref{equ1.1} under the low regularity of solution \eqref{equ1.4}, as demonstrated in Theorem \ref{nnthm5.2}.

	\section{Preliminaries: Numerical schemes}
In this section, we derive numerical discretization schemes for nonlocal diffusion problem \eqref{equ1.1}  and the corresponding steady-state problem \eqref{equ1.2} in $\Omega:=(a,b) \subseteq \mathbb{R}$, respectively.

Let the partition $\pi_h$ with the interval $(a,b)$
\begin{equation*}
   \pi_h: a =x_{0}<x_{1}<x_{2}<\dots<x_{2N-1}<x_{2N}=b.
\end{equation*}
 We focus on graded meshes as
  \begin{equation}\label{neq2.1}
  x_j=
  \begin{cases}	
  a+\frac{b-a}{2}\left( \frac{j}{N} \right)^r,   &    j=0,1,\dots ,N,\\
  b-\frac{b-a}{2}\left(2-\frac{j}{N} \right)^r,  &  	j=N+1,N+2,\dots ,2N
\end{cases}
  \end{equation}
with the grid sizes $h_j :=x_j-x_{j-1}$.
Here $r> 1$ is the grading exponent   of the  standard graded meshes. We extend it, for $0<r< 1$, called anomalous graded meshes. 

Let $S^{h}$ be the space of continuous piecewise-linear polynomials defined with respect to the partition $\pi_h$ and choose the standard hat functions as a basis which we denote as $\left\{\phi_{j }(x) \right\}_{0}^{2N}$.
Namely, the  piecewise linear  basis function is defined by
\begin{equation*}
\phi_j(x)=
\begin{cases}
    \frac{x-x_{j-1}}{x_j-x_{j-1}}, &  x \in \left[x_{j-1}, x_{j}\right],\\
 \frac{x_{j+1}-x}{x_{j+1}-x_j}, & x \in \left[x_{j}, x_{j+1}\right],\\
  0, &  {\rm otherwise},
\end{cases}
\end{equation*}
with $j=1,\cdots,2N-1$, 
\begin{equation*}
\phi_{0}(x)=
\begin{cases}
     \frac{x_{1}-x}{x_{1}-x_0}, & x \in \left[x_{0}, x_{1}\right],\\
     0, &  {\rm otherwise},
\end{cases}
~~~{\rm and}~~~
\phi_{2N}(x)=
\begin{cases}
    \frac{x-x_{2N-1}}{x_{2N}-x_{2N-1}}, &  x \in \left[x_{2N-1}, x_{2N}\right],\\
 0, &  {\rm otherwise}.
\end{cases}
\end{equation*}

\subsection{Collocation method for steady-state problems \eqref{equ1.2}}
Let $\Pi_{h} u(x)$ be the piecewise linear interpolation approximation for the solution $u(x)$ with  
homogeneous Dirichlet boundary conditions, namely,
\[
\Pi_{h} u(x)=\sum_{k=0}^{2N} u \left(x_{k}\right) \phi_{k}(x)=\sum_{k=1}^{2N-1} u \left(x_{k}\right) \phi_{k}(x).
\]

Taking it into \eqref{equ1.2} yields approximation for  the value of $-\mathcal{L} u(x_{i})$  given by
\begin{equation*}
\begin{split}
   -\mathcal{L}[\Pi_{h} u](x_{i}) 
   &= \int_{a}^{b} \frac{\sum_{k=1}^{2N-1} u \left(x_{k}\right) \phi_{k}(x_{i}) -\sum_{k=1}^{2N-1} u \left(x_{k}\right) \phi_{k}(y)} {\left|x_{i}-y\right|^{\alpha}} \, d{y}
     =\sum_{k=1}^{2N-1} a_{i,k} u \left(x_{k}\right)
\end{split}
    \end{equation*}
with   the coefficients 
\begin{equation}\label{neq2.2}
a_{i,k}:= \int_{a}^{b} \frac{ \phi_{k}(x_{i}) - \phi_{k}(y)} {\left|x_{i}-y\right|^{\alpha}} \, dy,~~~i,k=1,2,\cdots,2N-1.
  \end{equation}
Therefore, we can rewrite \eqref{equ1.2} as
    \begin{equation}\label{neq2.3}
      -\mathcal{L}[\Pi_{h} u](x_{i})  =f(x_{i}) +R_{i}, ~~~i=1,2,\cdots,2N-1,
    \end{equation}
with the local truncation error 
    \begin{equation}\label{nequ2.4}
      R_{i}=\mathcal{L}[u-\Pi_{h}u](x_{i})=\int_{a}^{b} {\frac{u(y)-\Pi_{h}u(y)}{|x_{i}-y|^{\alpha}}\,dy}.
    \end{equation}

Let $u_{i}$ be the approximation of $u(x_{i})$ and  $f_{i}:=f(x_{i})$. It  yields  the follow numerical scheme by \eqref{neq2.3}:
\begin{equation*}
  \mathcal{L}_h u_{i} =f_{i}~~~{\rm with}~~~\mathcal{L}_h u_{i}:=\sum_{k=1}^{2N-1}a_{i,k}u_{k}.
\end{equation*}
In particular, the above system of equations has the matrix form
\begin{equation}\label{nequ2.5}
AU=F,
\end{equation}
where  the coefficient matrix $A$ and the matrix of the grid function are defined by
\[
A=(a_{i,k})\in \mathbb{R}^{\left(2N-1\right)\times \left(2N-1\right)} , \ U=(u_{1},\dots,u_{2N-1})^T, \ F=(f_{1},\dots,f_{2N-1})^T.
\]

\subsection{Collocation method for time-dependent problems \eqref{equ1.1}}
In the time direction, we use  Crank-Nicolson scheme \cite[p.\,131]{Tho:06} on uniform meshes  with 
$t_{j}=j\tau$, $\tau=\frac{1}{M}$, $j=0,1,\cdots, M$. Here we mainly focus on the space direction, since the convergence rate with $\mathcal{O}\left( M^{-2 } \right)$ of the time discretization is well understood.

Let $u_i^j$  be the approximation of $u(x_{i},t_{j})$ and  $f_{i}^j=f(x_{i},t_j)$.
From \eqref{nequ2.5}, the full discretization of nonlocal diffusion problem  \eqref{equ1.1}
has the matrix form
\begin{equation}\label{nequ2.6}
\left(I+\frac{\tau }{2}A\right) U^{j}=\left(I-\frac{\tau }{2}A\right) U^{j-1}+ \tau F^{j-\frac{1}{2}},~~~j=1,2,\cdots,M,
\end{equation}
where  the coefficient matrix $A$ is given  in \eqref{nequ2.5} and the matrix of the grid function are defined by
\[
 U^{j}=(u_{1}^{j},\dots,u_{2N-1}^{j})^T, \ F^{j-\frac{1}{2}}=(f_{1}^{j-\frac{1}{2}},\dots,f_{2N-1}^{j-\frac{1}{2}})^T.
\]
 
\subsection{Spectral analysis for stiffness matrix $A$ in \eqref{nequ2.5}}
First, we  give some lemmas that will be used. 
\begin{lemma}\cite[p.\,28]{Quarteroni:07}\label{nnlem2.1}
A real matrix $A$ of order $n$ is positive definite  if and only if  its symmetric part $H=\frac{A+A^T}{2}$ is positive definite.
Let $H \in \mathbb{R}^{n\times n}$ be symmetric. Then $H$ is positive definite if and only if the eigenvalues of $H$ are positive.
\end{lemma}
\begin{lemma}\cite{Var:75}\label{nnlem2.2}
    Assume $A$ is diagonally dominant by rows. Then
    \begin{equation*}
        \left\|A^{-1} \right\|_{\infty} \leq \frac{1}{\delta}~~~{\rm with}~~~\delta=\min_{i} \left( \left|a_{i,i} \right| - \sum_{j\neq i} \left|a_{i,j}\right| \right).
    \end{equation*}
\end{lemma}
\begin{lemma}\label{nnlem2.3}
 Let the matrix $G_\alpha=(g_{i,j})\in \mathbb{R}^{\left(2N-1\right)\times \left(2N-1\right)} $ and its the element $g_{i,j}:=\int_{a}^{b} \frac{ \phi_{j}({y})} {\left|x_{i}-y\right|^{\alpha}} \, dy$. Then  $G_\alpha$ is a positive matrix and  $g_{i,j}$ is explicitly computed as 
\begin{equation*}
			g_{i,j}=\frac{1}{\left(1-\alpha\right)\left(2-\alpha\right)}C_{j}Q_j^i>0
		\end{equation*}
		with
		\begin{equation*}
			C_j= \left(\frac{1}{h_{j}},-\frac{1}{h_{j}}-\frac{1}{h_{j+1}},\frac{1}{h_{j+1}}\right)
			~~~{\rm and}~~~  Q_{j}^{i}=\left( \begin{array}{c}	\left| x_{j-1}-x_i \right|^{2-\alpha}\\	\left| x_j-x_i \right|^{2-\alpha}\\	\left| x_{j+1}-x_i \right|^{2-\alpha}\\\end{array} \right).
		\end{equation*}
\end{lemma}

\begin{proof}
Taking  $C_{\alpha}:=\frac{1}{\left(1-\alpha\right)\left(2-\alpha\right)}$,
we can check  $g_{i,i} 
       =C_{\alpha} \left(h_{i+1}^{1-\alpha} +h_{i}^{1-\alpha} \right)>0.
      $
On the other hand, for $j\neq i$, there exists 
 \begin{equation*}
 \begin{split}
    g_{i,j} 
      & =  C_{\alpha} 
   \left[\frac{\left|x_{i}-x_{j+1} \right|^{2-\alpha}}{h_{j+1}} 
           - \frac{h_{j}+h_{j+1}}{h_{j}h_{j+1}}\left|x_{i}-x_{j}\right|^{2-\alpha}  
             +\frac{\left|x_{i}-x_{j-1} \right|^{2-\alpha}}{h_{j}}\right]\\
      & =C_{\alpha} \frac{h_{j}+h_{j+1}}{h_{j}h_{j+1}}
      \left[\frac{h_{j} \left|x_{i}-x_{j+1} \right|^{2-\alpha}}{h_{j}+h_{j+1}} 
           - \left|x_{i}-x_{j}\right|^{2-\alpha}  
             +\frac{h_{j+1} \left|x_{i}-x_{j-1} \right|^{2-\alpha}}{h_{j}+h_{j+1}}\right].\\
  \end{split} 
 \end{equation*}
 Since 
$\frac{h_{j} }{h_{j}+h_{j+1}} \left|x_{i}-x_{j+1} \right| 
+\frac{h_{j+1} }{h_{j}+h_{j+1}}\left|x_{i}-x_{j-1} \right|
= \left|x_{i}-x_{j}\right|$ 
and $x \mapsto x^{2-\alpha}$ is a convex function for $x\geq 0$ under $0<\alpha<1$,
by Jensen's inequality we have 
\[\frac{h_{j} }{h_{j}+h_{j+1}} \left|x_{i}-x_{j+1} \right|^{2-\alpha}
             +\frac{h_{j+1} }{h_{j}+h_{j+1}}\left|x_{i}-x_{j-1} \right|^{2-\alpha}
             >\left|x_{i}-x_{j}\right|^{2-\alpha},~j\neq i.
\]
The proof is completed. 
\end{proof}

From \eqref{neq2.2}, \eqref{nequ2.5} and  Lemma \ref{nnlem2.3},   the entries of the stiffness matrix $A=(a_{ij})\in\mathbb{R}^{\left(2N-1\right)\times \left(2N-1\right)}$ with $\alpha \in (0,1)$ can be explicitly evaluated by
\begin{equation}\label{nequ2.7}
  A=D_\alpha-G_\alpha,
\end{equation}
where the diagonal matrix $D_\alpha$ is defined by
\begin{equation*}
    D_\alpha={\rm diag}\left( d_{1},d_{2},\cdots,d_{2N-1}\right)~~{\rm with}~~d_{i}:=\int_{a}^{b}{\frac{dy}{|x_{i}-y|^{\alpha}}}
\end{equation*}
 for $i=1,2,\cdots,2N-1$.

\begin{lemma}\label{nnlem2.4}
Let the matrix $A$ be defined by \eqref{nequ2.5}. Then  $A$ is a strictly diagonally dominant matrix by rows with positive entries on the diagonal and nonpositive off-diagonal entries.
Moreover, the linear equation \eqref{nequ2.5} has a unique solution. 
\end{lemma}
\begin{proof}
It is evident to observe that
\begin{equation}\label{nequ2.8}
\begin{split}
   a_{ii}   &  =\int_{x_{i-1}}^{x_{i+1}}{\frac{1-\phi_{i}\left(y\right)}{\left|x_{i}-y\right|^{\alpha}} \, dy}\\
   & =\frac{1}{1-\alpha}\left[(x_{i}-a)^{1-\alpha}+(b-x_{i})^{1-\alpha} \right] - \frac{1}{\left(1-\alpha\right)\left(2-\alpha\right)} \left(h_{i}^{1-\alpha} + h_{i+1}^{1-\alpha} \right) >0\\
\end{split}
\end{equation}
and $a_{i,j}=-g_{i,j}<0$ for  $j\neq i$ by Lemma \ref{nnlem2.3}.

On the other hand, for $1\leq i\leq 2N-1$, using Taylor series expansion, we have 
\begin{equation}\label{nequ2.9}
  \begin{split}
     \sum_{j=1}^{2N-1}a_{i,j}  
     & = \int_{a}^{b}{\frac{dy}{|x_{i}-y|^{\alpha}}} -\sum_{j=1}^{2N-1}\int_{x_{j-1}}^{x_{j+1}}  {\frac{\phi_{j}\left(y\right)}{\left|x_{i}-y\right|^{\alpha}}\, dy} \\
       &= \int_{x_{0}}^{x_{1}}  {\frac{\phi_{0}\left(y\right)}{\left|x_{i}-y\right|^{\alpha}}\, dy} + \int_{x_{2N-1}}^{x_{2N}}  {\frac{\phi_{2N}\left(y\right)}{\left|x_{i}-y\right|^{\alpha}}\, dy}\\
        &= \frac{1}{1-\alpha} \left[\left(x_{i}-x_{0}\right)^{1-\alpha} 
        - \frac{\left( x_{i}-x_{0}\right)^{2-\alpha}-\left( x_{i}-x_{1}\right)^{2-\alpha}}{\left(2-\alpha\right)h_{1}} \right]\\
       &\quad+\frac{1}{1-\alpha} \left[\left(x_{2N}-x_{i} \right)^{1-\alpha}
       -\frac{\left( x_{2N}-x_{i}\right)^{2-\alpha}-\left( x_{2N-1}-x_{i}\right)^{2-\alpha}}{\left(2-\alpha\right)h_{2N}} \right]\\
      & \geq \frac{1}{2}\left[  h_{1} \left(x_{i}-x_{0}\right)^{-\alpha} + h_{2N} \left(x_{2N}-x_{i}\right)^{-\alpha}\right]>0,
  \end{split}
\end{equation}
which implies that the matrix $A$ is M-matrix.
From Theorem 1.21 of \cite{Varga:00}, the matrix $A$ is nonsigular. The proof is completed. 
\end{proof}


Let the condition number $\kappa_{p}= \left\|A\right\|_{p} \left\|A^{-1}\right\|_{p}$ with $p=1,2,\cdots,\infty$. Then there is
\begin{lemma}
    Let the matrix $A$ is defined as \eqref{nequ2.5}. Then the condition number
    \begin{equation*}
        \kappa_{\infty}=\left\|A\right\|_{\infty} \left\|A^{-1}\right\|_{\infty} = \mathcal{O} \left(N^{r}\right),
    \end{equation*}
where $r$ is the grading exponent of  graded meshes.
\end{lemma}

\begin{proof}
From Lemma \ref{nnlem2.4} and \eqref{nequ2.8}, we have
\begin{equation*}
    \left\|A\right\|_{\infty}\leq 2 \max_{1\leq i \leq 2N-1} a_{i,i}  \leq \frac{4}{1-\alpha} \left(b-a\right)^{1-\alpha}.
\end{equation*}
On the other hand, it yields $\left\|A^{-1}\right\|_{\infty} \leq 2 \left(\frac{b-a}{2}\right)^{\alpha} N^{r}$ by Lemma \ref{nnlem2.2},
since 
    \begin{equation*}
        \begin{split}
            \left|a_{i,i} \right| - \sum_{j\neq i} \left|a_{i,j}\right|
            & =\sum_{j=1}^{2N-1} a_{i,j}
             \geq \frac{1}{2} \left(\frac{b-a}{2}\right)^{-\alpha} N^{-r}. 
        \end{split}
    \end{equation*}
The proof is completed.
\end{proof}

\begin{remark}
From   Lemma \ref{nnlem2.4} and Theorem $1.21$ of \cite[p.\,23]{Varga:00}, it yields  
\[\Re\left(\lambda(A)\right)>0\] and  $A$  nonsingular for $r>0$. In particular, the matrix $A$ is symmetric positive-definite for uniform meshes with $r=1$. 

However, from Lemma \ref{nnlem2.1} and counter-example in Table \ref{table2.1}, it shows that
$$\min\left( \lambda\left(H\right) \right)<0,~~\max\left( \lambda\left(H\right) \right)>0~~{\rm with}~~H=\frac{A+A^T}{2}.$$
 That is,
the matrix $A$ may be a nonsymmetric and indefinite for  graded meshes.
\end{remark}

    \begin{table}[h]
\centering
\caption{Maximum  and minimum eigenvalues of $H=\frac{A+A^{T}}{2}$ with $\alpha=0.5$, $N=500$}\label{table2.1}
\begin{tblr}{
  cells = {c},
  hlines,
  vlines,
}
$r$                & 0.2     & 0.9     & 1      & 1.1     & 4       \\
$\max(\lambda(H))$ & 8.8728  & 5.5807  & 5.5766 & 5.5727  & 5.4974  \\
$\min(\lambda(H))$ & -4.6418 & -0.0023 & 0.0039 & -0.0046 & -0.8304 
\end{tblr}
\end{table}

\section{Error analysis: Steady-state problems with smooth solution}
In this section,   we first clarify that standard graded meshes perform worse than the uniform grid, and may even lead to divergence for the steady-state counterpart, despite the solution being sufficiently smooth. However,  optimal convergence rates  arise in anomalous graded meshes.

Without loss of generality,  we take $\Omega=\left( 0,2T\right)$ and  rewrite \eqref{neq2.1} as
\begin{equation}\label{eqn3.1}
x_j= \begin{cases}
T\left( \frac{j}{N} \right) ^r,	     &{\rm for}~~ j=0,1,\dots ,N,\\
2T-T\left( 2-\frac{j}{N} \right) ^r, &{\rm for}~~ j=N+1,N+2,\dots ,2N.
\end{cases}
\end{equation}
From the mean value theorem and the definition of $\{h_j\}_{j=0}^{2N}$, it follows that \cite{CDMSM:23,SOG:17}
\begin{equation}\label{eqn3.2}
h_j =x_j-x_{j-1} \le
	\begin{cases}
	CN^{-r}j^{r-1},        \ &{\rm for}~~ j=1,\dots, N, \\
	CN^{-r}(2N+1-j)^{r-1}, \ &{\rm for}~~ j=N+1,\dots, 2N.
	\end{cases}
\end{equation}
Note that $h_{j}\leq CN^{-1}$ for standard graded meshes $r\geq1$; however, $h_1>{N}^{-1}$ for anomalous graded meshes $0<r<1$.

\emph{Notation.} Throughout this article and above,  $C$ denotes a positive constant, not necessarily the same at different occurrences,  that is independent of~$N$ and of any index such as~$i$ or~$j$.
For any real number $s\in \mathbb{R}$, $\lceil s \rceil$ represents the smallest integer that is not less than~$s$.

\subsection{Local truncation error}
We next study the local truncation error  for \eqref{nequ2.4} under the smooth solution. From \eqref{nequ2.4}, we have
\begin{equation}\label{eqn3.3}
  R_{i}= \sum_{k=1}^{2N} {\int_{x_{k-1}}^{x_{k}} \frac{u \left( y \right) -\Pi_h u \left( y \right)}{\left|x_i -y \right|^{\alpha}} \,dy} =\sum_{k=1}^{2N} {\mathcal{T}_{i,k}}
\end{equation}
 with
 \begin{equation}\label{eqn3.4}
   \mathcal{T}_{i,k} := \int_{x_{k-1}}^{x_{k}} \frac{u \left( y \right) -\Pi_h u \left( y \right)}{\left|x_i -y \right|^{\alpha}} \,dy.
 \end{equation}

\begin{lemma}\label{nnlem3.1}
If $u(x)\in C^{2}(\bar{\Omega})$ and $r\geq 1$, then there exists a constant $C$ such that
\begin{equation*}
 \left| R_{i} \right| \leq C N^{-2}
\end{equation*}
for $i=1,2,\cdots,2N-1$.
\end{lemma}

\begin{proof}
  Since $u(x)\in C^{2}(\bar{\Omega})$, by mean value theorem, there exists 
  \begin{equation*}
    \begin{split}
      \left| R_{i} \right| 
      & \leq \sum_{k=1}^{2N} \left| \mathcal{T}_{i,k}  \right|\\
         & \leq C \sum_{k=1}^{2N}h_{k}^{2} \left(\max_{s\in[x_{k-1}, x_{k}]} \left| u_{xx}(s)\right|\right) 
         \int_{x_{k-1}}^{x_{k}}  \left| x_{i}-y \right|^{-\alpha} \,dy\\
         & \leq C {N}^{-2} \int_{x_{0}}^{x_{2N}}  \left| x_{i}-y \right|^{-\alpha} \,dy \leq CN^{-2},
    \end{split}
  \end{equation*}
 for $i=1,2,\cdots,2N-1$. The proof is completed. 
\end{proof}

The conclusion of Lemma \ref{nnlem3.1} may not hold for the anomalous graded meshes $r<1$ due to the presence of $h_{1}>N^{-1}$.

\begin{lemma}\label{nnlem3.2}
If $u(x)\in C^{2}(\bar{\Omega})$ and $r<1$, then there exists a constant $C$ such that
\begin{equation*}
 \sum_{k=1}^{N} \left|\mathcal{T}_{i,k}\right| \leq
   \begin{cases}
         C  N^{-r\left(3-\alpha\right)} i^{-r\alpha},       & {\rm if}~~3r-2<0, \\
         C  N^{-r\left(3-\alpha\right)} i^{-r\alpha} \ln N, & {\rm if}~~3r-2=0, \\
         C  N^{-\left(2-r\alpha\right)} i^{-r\alpha},       & {\rm if}~~3r-2>0
         \end{cases}
\end{equation*}
for $i=1,2,\cdots,N$.
\end{lemma}
\begin{proof}
  Since $u\in C^{2}\left(\bar{\Omega}\right)$, there exists a constant $C$ such that 
  \begin{equation}\label{eqn3.5}
   \left| \mathcal{T}_{i,k}\right|
   \leq C \int_{x_{k-1}}^{x_{k}} {\frac{\left(y-x_{k-1}\right)  \left(x_{k}-y\right)}{\left|x_{i}-y\right|^{\alpha}}} \, dy 
   \leq C h_{k}^{2} \int_{x_{k-1}}^{x_{k}} {\frac{1}{\left|x_{i}-y\right|^{\alpha}}} \, dy. 
  \end{equation}
  
 We next estimate  the local truncation error.  
   From \eqref{eqn3.5}, one has
  \begin{equation*}
    \left|\mathcal{T}_{1,1}\right|
    \leq C \int_{x_{0}}^{x_{1}} {y\left(x_{1}-y\right)^{1-\alpha}} \, dy \leq C x_{1}^{3-\alpha} =C N^{-r\left(3-\alpha\right)},
  \end{equation*}
and
  \begin{equation*}
  \begin{split}
     \left|\mathcal{T}_{i,1}\right|
       & \leq C h_{1}^{3} \left(x_{i}-x_{1}\right)^{-\alpha}  
        \leq C N^{-r\left(3-\alpha\right)} i^{-r\alpha},~~i>1.
  \end{split}   
  \end{equation*}
  
  For $1<k\leq\lceil\frac{i}{2}\rceil$, using  \eqref{eqn3.5} and \eqref{eqn3.2}, we calculate 
  \begin{equation*}
    \begin{split}
       \sum_{k=2}^{\lceil\frac{i}{2}\rceil} \left|\mathcal{T}_{i,k}\right| 
       & \leq C \sum_{k=2}^{\lceil\frac{i}{2}\rceil}  h_{k}^{3} \left(x_{i}-x_{k}\right)^{-\alpha }
       \leq C \sum_{k=2}^{\lceil\frac{i}{2}\rceil}  h_{k}^{3} \left(x_{i}-x_{\lceil\frac{i}{2}\rceil}\right)^{-\alpha }\\
         & \leq C \sum_{k=2}^{\lceil\frac{i}{2}\rceil}  N^{-r\left(3-\alpha\right)} i^{-r\alpha} k^{3r-3}\\
         & \leq 
         \begin{cases}
         C  N^{-r\left(3-\alpha\right)} i^{-r\alpha} ,         & {\rm if}~~3r-2<0, \\
         C  N^{-r\left(3-\alpha\right)} i^{-r\alpha} \ln i,    & {\rm if}~~3r-2=0, \\
         C  N^{-r\left(3-\alpha\right)} i^{-r\alpha} i^{3r-2}, & {\rm if}~~3r-2>0.
         \end{cases}
    \end{split}
  \end{equation*}

   For $\lceil\frac{i}{2}\rceil<k<i$, we also obtain
  \begin{equation*}
    \begin{split}
       \sum_{k=\lceil\frac{i}{2}\rceil+1}^{i-1}  \left|\mathcal{T}_{i,k}\right|
        & \leq C \left(N^{-r}i^{r-1}\right)^{2} \int_{x_{\lceil\frac{i}{2}\rceil}}^{x_{i-1}} {\left(x_{i}-y\right)^{-\alpha}} \, dy \\
         & \leq C \left(N^{-r}i^{r-1}\right)^{2} \left(x_{i}-x_{\lceil\frac{i}{2}\rceil}\right)^{1-\alpha} \\
         &  \leq C N^{-r\left(3-\alpha\right)} i^{r\left(3-\alpha\right)-2}.
    \end{split}
  \end{equation*}
  We  deduce the following, for $i>1$,
  \begin{equation*}
    \begin{split}
       \left| \mathcal{T}_{i,i} \right|
        & \leq C h_{i}^{3-\alpha} 
        \leq N^{-r\left(3-\alpha\right)} i^{r\left(3-\alpha\right)-\left(3-\alpha\right)}, \\
    \end{split}
  \end{equation*}
  and
  \begin{equation*}
    \begin{split}
       \left|\mathcal{T}_{i,i+1}\right|
       \leq C h_{i+1}^{3-\alpha} 
       \leq C N^{-r\left(3-\alpha\right)} i^{r\left(3-\alpha\right)-\left(3-\alpha\right)}.
    \end{split}
  \end{equation*}
  
  Let ${J}=\min\left\{2i,N\right\}$. From \eqref{eqn3.5} and \eqref{eqn3.2}, we have
  \begin{equation*}
    \begin{split}
       \sum_{k=i+2}^{J} \left|\mathcal{T}_{i,k}\right|
       & \leq C \sum_{k=i+2}^{J} h_{k}^{2} \int_{x_{k-1}}^{x_{k}} {\left(y-x_{i}\right)^{-\alpha}} \, dy \\
         & \leq C N^{-2r} i^{2r-2} \int_{x_{i}}^{x_{J}} {\left(y-x_{i}\right)^{-\alpha}} \, dy\\
         & \leq C N^{-r\left(3-\alpha\right)} i^{r\left(3-\alpha\right)-2},
    \end{split}
  \end{equation*}
and for the special case $J=2i<N$, it yields 
  \begin{equation*}
    \begin{split}
       \sum_{k=J+1}^{N} \left|\mathcal{T}_{i,k}\right|
       & \leq C\sum_{k=J+1}^{N}  h_{k}^{3} \left(x_{k-1}-x_{i}\right)^{-\alpha} \\
         &\leq C \sum_{k=J+1}^{N} N^{-r\left(3-\alpha\right)} i^{-r\alpha}  k^{3r-3}\\
         & \leq  
         \begin{cases}
          CN^{-r\left(3-\alpha\right)} i^{-r\alpha},       & {\rm if}~~3r-2<0, \\
          CN^{-r\left(3-\alpha\right)} i^{-r\alpha}\ln N,  & {\rm if}~~3r-2=0, \\
          CN^{-\left(2-r\alpha\right)} i^{-r\alpha},       & {\rm if}~~3r-2>0.
          \end{cases}
    \end{split}
  \end{equation*}
The proof is completed.
\end{proof}

\begin{lemma}\label{nnlem3.3}
If $u(x)\in C^{2}(\bar{\Omega})$ and $r<1$, then there exists a constant $C$ such that
\begin{equation*} 
 \sum_{k=N+1}^{2N} \left|\mathcal{T}_{i,k}\right| \leq
   \begin{cases}
         C  N^{-r\left(3-\alpha\right)} i^{-r\alpha},       & {\rm if}~~3r-2<0, \\
         C  N^{-r\left(3-\alpha\right)} i^{-r\alpha} \ln N, & {\rm if}~~3r-2=0, \\
         C  N^{-\left(2-r\alpha\right)} i^{-r\alpha},       & {\rm if}~~3r-2>0
         \end{cases}
\end{equation*}
for $i=1,2,\cdots,N$.
\end{lemma}
\begin{proof}
    For $N+1\leq k \leq 2N-i$, taking $\zeta =2T-y$ and using \eqref{eqn3.1}, \eqref{eqn3.2}, \eqref{eqn3.5},   it yields 
    \begin{equation*}
       \begin{split}
        \left|T_{i,k} \right| &\leq C h_{k}^{2} \int_{x_{k-1}}^{x_{k}} \left(y-x_{i}\right)^{-\alpha} \, dy
        = C h_{2N-k+1}^{2} \int_{x_{k-1}}^{x_{k}} \left(y-2T+2T-x_{i}\right)^{-\alpha} \, dy\\
        &  =C h_{2N-k+1}^{2} \!\int_{x_{2N-k}}^{x_{2N-k+1}} \!\! \left(2T-x_{i} -\zeta\right)^{-\alpha} \, d\zeta
        \leq C h_{2N-k+1}^{2} \!\int_{x_{2N-k}}^{x_{2N-k+1}} \!\! \left(\zeta-x_{i}\right)^{-\alpha}\, d\zeta,
    \end{split} 
    \end{equation*}
    since $2T-x_{i}-\zeta \geq \zeta-x_{i}\geq0$ when $\zeta\in[x_{2N-k},x_{2N-k+1}]$.

  On the other hand, taking $\zeta =2T-y$, for $k>2N-i$,  there exits
    \begin{equation*}
         \begin{split}
        \left|T_{i,k} \right|
        & \leq C h_{2N-k+1}^{2} \! \int_{x_{2N-k}}^{x_{2N-k+1}} \!\!\left(2T-x_{i} -\zeta\right)^{-\alpha} \, d\zeta 
        \leq C h_{2N-k+1}^{2} \! \int_{x_{2N-k}}^{x_{2N-k+1}} \!\!\left(x_{i}-\zeta\right)^{-\alpha}\, d\zeta,
    \end{split} 
    \end{equation*}
    because $2T-x_{i}-\zeta \geq x_{i}-\zeta\geq0$ when $\zeta\in[x_{2N-k},x_{2N-k+1}]$.
    The similar arguments can be performed as Lemma \ref{nnlem3.2}, the desired results is obtained. 
\end{proof}

\begin{lemma}\label{nnlem3.4}
    If $u(x)\in C^{2}(\bar{\Omega})$ and $r<1$, then there exists a constant  $C$ such that 
\begin{equation*}
 \left|R_{i}\right| \leq \sum_{k=1}^{2N} \left|\mathcal{T}_{i,k}\right| \leq
   \begin{cases}
      C N^{-r\left(3-\alpha\right)} i^{-r\alpha},       & {\rm if}~~3r-2< 0, \\
      C N^{-r\left(3-\alpha\right)} i^{-r\alpha} \ln N, & {\rm if}~~3r-2= 0, \\
      C N^{-\left(2-r\alpha\right)} i^{-r\alpha},       & {\rm if}~~3r-2> 0
      \end{cases}
\end{equation*}
for $i=1,2,\cdots,N$, and
\begin{equation*}
 \left|R_{i}\right| \leq \sum_{k=1}^{2N} \left|\mathcal{T}_{i,k}\right| \leq
   \begin{cases}
      C N^{-r\left(3-\alpha\right)} \left(2N- i\right)^{-r\alpha},       & {\rm if}~~3r-2< 0, \\
      C N^{-r\left(3-\alpha\right)} \left(2N- i\right)^{-r\alpha} \ln N, & {\rm if}~~3r-2= 0, \\
      C N^{-\left(2-r\alpha\right)} \left(2N- i\right)^{-r\alpha},       & {\rm if}~~3r-2> 0
      \end{cases}
\end{equation*}
for $i=N+1,N+2,\cdots,2N-1$.
\end{lemma}
\begin{proof}
The first part of this lemma can be obtained by  Lemmas  \ref{nnlem3.2} and \ref{nnlem3.3}. For second part can be similarly proved, we omit it here.
\end{proof}

\subsection{Convergence analysis}
Let $e_{k}:=u\left(x_{k}\right)-u_{k}$ with $e_{0}=e_{2N}=0$ and 
  $e=\left(e_{1},e_{2},\cdots,e_{2N-1}\right)^T.$ From \eqref{neq2.3} and \eqref{nequ2.5}, it leads to 
  \begin{equation}\label{eq4.2}
    \sum_{k=1}^{2N-1} a_{i,k} e_{k}= R_{i}~~~{\rm for} ~~~i,k=1,2,\cdots,2N-1.
  \end{equation}
  
Let $\left|e_{i_{0}}\right|:=\| e\|_\infty=\max_{1\leq k\leq 2N-1 } \left|e_{k}\right|$.
By Lemma \ref{nnlem2.4}, we have
   \begin{equation}\label{eqn3.6}
     \begin{split}
      \left|R_{i_{0}}\right|
       & = \left|a_{i_{0},i_{0}}e_{i_{0}} +\sum_{k=1,k\neq i_{0}}^{2N-1} a_{i_{0},k}e_{k} \right| \geq a_{i_{0},i_{0}} \left|e_{i_{0}}\right|- \sum_{k=1,k\neq i_{0}}^{2N-1} \left|a_{i_{0},k}\right| \left|e_{k}\right|  \\
        & \geq \left(\sum_{k=1}^{2N-1} a_{i_{0},k}\right) \left|e_{i_{0}}\right|.
   \end{split} 
   \end{equation}
 Then we have the following result.
\begin{theorem}\label{nnthm3.5}
  If $u(x)\in C^{2}\left(\bar{\Omega}\right)$ and $u_{k} $ is the approximate solution of $u(x_{k})$ computed by the discretization scheme \eqref{nequ2.5}, then
  \begin{equation}\label{eq4.1}
    \max_{1\leq i \leq 2N-1} |u\left(x_{i}\right)-u_{i}| 
\leq
\begin{cases}
  C N^{r-2},      & {\rm if}~~r>\frac{2}{3},  \\
  C N^{r-2}\ln N, & {\rm if}~~r=\frac{2}{3},\\
  C N^{-2r},      & {\rm if}~~0<r<\frac{2}{3}.
\end{cases}
  \end{equation}
\end{theorem}
\begin{proof}      
From   \eqref{nequ2.9}, \eqref{eqn3.1}, \eqref{eqn3.6} and Lemma  \ref{nnlem3.1}, we  know that, for $r\geq1$,
      \begin{equation*}
        \left\|e\right\|_{\infty} \leq \frac{\left|R_{i_{0}}\right|}{\sum_{k=1}^{2N-1}a_{i_{0},k}}
        \leq C \frac{N^{-2}}{N^{-r}\left(\left(x_{i_{0}}-x_{0}\right)^{-\alpha} +  \left(x_{2N}-x_{i_{0}}\right)^{-\alpha}\right)} 
        \leq C N^{r-2},
      \end{equation*}
since 
$\left(x_{i_{0}}-x_{0}\right)^{-\alpha} +  \left(x_{2N}-x_{i_{0}}\right)^{-\alpha}\geq 2T^{-\alpha}$.

We now consider the case for anomalous graded meshes $0<r<1$.
According to \eqref{nequ2.9},  \eqref{eqn3.6} and Lemma  \ref{nnlem3.4}, we have,
      for $1\leq i_{0} \leq N$, 
      \begin{equation*}
          \left\|e\right\|_{\infty} 
          \leq \frac{\left|R_{i_{0}}\right|}{\sum_{k=1}^{2N-1}a_{i_{0},k}}
          \leq C\frac{\left|R_{i_{0}}\right|}{h_1 \left(x_{i_{0}}-x_{0}\right)^{-\alpha}}
          \leq \begin{cases}
  C N^{r-2},      & {\rm if}~~r>\frac{2}{3},  \\
  C N^{r-2}\ln N, & {\rm if}~~r=\frac{2}{3},\\
  C N^{-2r},      & {\rm if}~~0<r<\frac{2}{3};
\end{cases}
      \end{equation*}
      and, 
      for $N+1\leq i_{0} \leq 2N-1$, 
      \begin{equation*}
          \left\|e\right\|_{\infty} 
          \leq C\frac{\left|R_{i_{0}}\right|}{h_{2N} \left(x_{2N}-x_{i_{0}}\right)^{-\alpha}}
          \leq \begin{cases}
  C N^{r-2},      & {\rm if}~~r>\frac{2}{3}, \\
  C N^{r-2}\ln N, & {\rm if}~~ r=\frac{2}{3},\\
  C N^{-2r},      & {\rm if}~~0<r<\frac{2}{3}.
\end{cases}
      \end{equation*}
      The proof is completed.
\end{proof}

\section{Error analysis: Steady-state problems with low regularity solution}
 In the previous section, we clarified  that an optimal convergence rate emerges from anomalous graded meshes when considering the smooth solution. Conversely, in this section, a sharp error estimate appears in standard graded meshes due to the low regularity of the solution.

 Without loss of generality,  we take $\Omega=\left( 0,2T\right)$ and  rewrite \eqref{equ1.4} as
 \begin{equation}\label{eqn4.1}
	\begin{split}
		\left|\frac{\partial^{\ell} }{\partial x ^{\ell}} u(x)\right|  
		& \le C[x(2T-x)]^{\sigma-\ell}, ~0<\sigma<1 \ \text{ for }  \ell=0,1,2. 
	\end{split}
\end{equation}
Throughout the rest of the paper, we study the local truncation error and the global error under the low  regularity solution \eqref{eqn4.1} for nonlocal diffusion problems.

\subsection{Local truncation error}


By a standard error estimate for linear interpolation, since $u\in C^2(0,2T)$ in \eqref{eqn4.1}, we estimate \eqref{eqn3.4} as
\begin{equation}\label{eqn4.2}
  \left|\mathcal{T}_{i,k}  \right| \leq C h_{k}^{2} \left(\max_{s\in[x_{k-1}, x_{k}]} \left| u_{xx}(s )\right|\right) \int_{x_{k-1}}^{x_{k}}  \left| x_{i}-y \right|^{-\alpha} \,dy,~{\rm for }~k \ne 1,2N.
\end{equation}

\begin{lemma}\label{nnlem4.1}
Let $r>0$  and $0<\sigma <1$. Then  there exists a constant $C$ such that
	\[
	\sum_{k=1}^{i}{|\mathcal{T}_{i,k}|} \le
  \begin{cases}
C  N^{-r \left(1+\sigma -\alpha \right)}  i^{-r\alpha},                          &{\rm if}~~  r\left(1+\sigma\right)<2, \\          
C  N^{-r \left(1+\sigma -\alpha \right)}  i^{-r\alpha} \ln i,                    &{\rm if}~~  r\left(1+\sigma\right)=2, \\          
C  N^{-r \left(1+\sigma -\alpha \right)}  i^{r \left(1+\sigma-\alpha \right)-2}, &{\rm if}~~  r\left(1+\sigma\right)>2
	\end{cases}
	\]
	for all $i\in\{1, \dots, N\}$.
\end{lemma}
\begin{proof}
  Let $i\in\{1, \dots, N\}$ be arbitrary but fixed.
	Consider separately the cases $k=1=i,\ k=1<i, \ 1< k=i$ and  $1<k<i$.
  
From \eqref{eqn3.4} and \eqref{eqn4.1}, it can be computed 
\begin{equation}\label{eqn4.3}
  \begin{split}
     \left|\mathcal{T}_{1,1} \right| 
       &\leq \int_{x_{0}}^{x_{1}} \left|u(y)\right|\left(x_{1}-y\right)^{-\alpha} \, dy
       +\int_{x_{0}}^{x_{1}} \left|\Pi_{h}u(y)\right|\left(x_{1}-y\right)^{-\alpha} \, dy\\
       & \leq C \left[ \int_{x_{0}}^{x_{1}}{y^{\sigma}\left(x_{1}-y \right)^{-\alpha} \, dy} 
       +  \frac{x_{1}^{\sigma}}{h_{1}} \int_{x_0}^{x_1} {y \left(x_{1}-y \right)^{-\alpha}} \,dy \right] \\
       & \leq C x_{1}^{1+\sigma-\alpha} \leq N^{-r \left(1+\sigma-\alpha \right)},
  \end{split}
\end{equation}
and, for $i>1$,
\begin{equation}\label{eqn4.4}
  \begin{split}
     \left| \mathcal{T}_{i,1} \right|
       & \leq C \left(x_{i}-x_{1} \right)^{-\alpha}\int_{x_{0}}^{x_{1}} \left| u \left(y\right) -\Pi_h u\left(y\right) \right| \, dy 
       \\
       & \leq C \left(x_{1}^{\sigma+1} + x_{1}^{\sigma}h_{1}\right)
       \left(N^{-r} i^{r}\right)^{-\alpha} \leq C N^{-r \left(1+\sigma-\alpha\right)} i^{-r\alpha}. 
  \end{split}
\end{equation}

For $1<k \le \left\lceil \frac{i}{2} \right\rceil$, using \eqref{eqn3.2}, \eqref{eqn4.2} and the well-known convergence properties of the series $\sum_{j=2}^\infty j^{ \mu }$ $ (\mu \in\mathbb{R})$ , it yields 
\begin{equation}\label{eqn4.5}
\begin{split}
  \sum_{k=2}^{\left\lceil i/2 \right\rceil}\left|\mathcal{T}_{i,k}\right| 
     & \leq C \sum_{k=2}^{\left\lceil i/2 \right\rceil}  h_{k}^{3}  x_{k-1}^{\sigma-2} \left( x_{i} - x_{k}\right)^{-\alpha} \\
     & \leq C \sum_{k=2}^{\left\lceil i/2 \right\rceil}  \left(N^{-r}k^{r-1}\right)^{3} \left(N^{-r} k^{r}\right)^{\sigma-2}  \left(N^{-r} i^{r} \right)^{-\alpha}\\
     & \leq C N^{-r \left(1+\sigma -\alpha \right)} i^{-r\alpha} \sum_{k=2}^{\left\lceil i/2 \right\rceil} k^{r\left(1+\sigma\right)-3}\\
     &\leq \begin{cases}
C  N^{-r \left(1+\sigma -\alpha \right)}  i^{-r\alpha},
&{\rm if}~~  r\left(1+\sigma\right)<2, \\
C  N^{-r \left(1+\sigma -\alpha \right)}  i^{-r\alpha} \ln i, 
&{\rm if}~~  r\left(1+\sigma\right)=2, \\
C  N^{-r \left(1+\sigma -\alpha \right)}  i^{r \left(1+\sigma-\alpha \right)-2}, 
&{\rm if}~~  r\left(1+\sigma\right)>2.
	\end{cases}
\end{split}
\end{equation}

For $ \left\lceil \frac{i}{2} \right\rceil <k<i $,  using  \eqref{eqn4.2} and \eqref{eqn3.2}, there  exist
\begin{equation}\label{eqn4.6}
  \begin{split}
     \sum_{k=\lceil i/2 \rceil +1}^{i-1}{\left| \mathcal{T}_{i,k} \right|} 
  & \leq C \sum_{k=\lceil i/2 \rceil +1}^{i-1}  h_{k}^{2} x_{k-1}^{\sigma-2} \int_{x_{k-1}}^{x_{k}}{\left(x_{i}-y \right)^{-\alpha} \, dy} \\
  & \leq C \left(N^{-r} i^{r-1} \right)^{2} \left(N^{-r} i^{r}\right)^{\sigma-2}  \int_{x_{\lceil i/2 \rceil}}^{x_{i-1}}{\left(x_{i}-y \right)^{-\alpha} \, dy} \\
  & \leq C \left(N^{-r} i^{r-1} \right)^{2} \left(N^{-r} i^{r}\right)^{\sigma-2}   \left[\left(x_{i}-x_{\lceil i/2 \rceil}\right)^{1-\alpha} -h_{i}^{1-\alpha} \right]  \\
  & \leq C N^{-r\left(1+\sigma-\alpha \right)} i^{r \left(1+\sigma-\alpha\right) - 2},\\
  \end{split}
\end{equation}
and, for $i>1$, 
\begin{equation}\label{eqn4.7}
\begin{split}
   \left|\mathcal{T}_{i,i} \right| 
   & \leq C h_{i}^{2} x_{i-1}^{\sigma-2} \int_{x_{i-1}}^{x_{i}} {\left(x_{i}-y \right)^{-\alpha} \, dy} \\
     & \leq C h_{i}^{3-\alpha} x_{i}^{\sigma-2}\\
     & \leq C N^{-r \left(1+\sigma-\alpha\right)} i^{r \left(1+\sigma-\alpha\right) - \left( 3-\alpha \right)}.
\end{split}
\end{equation}
Combining bounds \eqref{eqn4.3}, \eqref{eqn4.4}, \eqref{eqn4.5},  \eqref{eqn4.6} and \eqref{eqn4.7} 
yields the desired result.
\end{proof}

\begin{lemma}\label{nnlem4.2}
Let $r>0$  and $0<\sigma <1$. Then  there exists a constant $C$ such that
	\[
	\sum_{k=i+1}^N |\mathcal{T}_{i,k}| 
\le \begin{cases}
C  N^{-r \left(1+\sigma -\alpha \right)} i^{-r\alpha},        &{\rm if }~~  r \left(1+\sigma\right)  <2, \\
C  N^{-r \left(1+\sigma -\alpha \right)} i^{-r\alpha}  \ln N, &{\rm if }~~  r \left(1+\sigma\right)  =2, \\
C  N^{-\left(2-r\alpha\right)}           i^{-r\alpha},        &{\rm if }~~  r \left(1+\sigma\right)  >2         
	\end{cases}
	\]
	for all $i\in\{1, \dots, N-1\}$.
\end{lemma}
\begin{proof}
 From \eqref{eqn4.2}, \eqref{eqn3.1} and \eqref{eqn3.2}, we can figure out 
\begin{equation}\label{eqn4.8}
\begin{split}
   \left|\mathcal{T}_{i,i+1} \right|
    & \leq C h_{i+1}^{2} x_{i}^{\sigma-2} \int_{x_{i}}^{x_{i+1}} {\left(y - x_{i} \right)^{-\alpha} \, dy} \\
     & \leq C h_{i+1}^{3-\alpha} x_{i}^{\sigma-2}\\
     & \leq C N^{-r \left( 1+\sigma -\alpha\right)} i^{r \left( 1+\sigma -\alpha\right) - \left( 3-\alpha \right)}.
\end{split}
\end{equation}

Taking $K =\min \{2i, N\}$ and using \eqref{eqn4.2}, \eqref{eqn3.2}, it leads to
\begin{equation}\label{eqn4.9}
\begin{split}
   \sum_{k=i+2}^{K}{\left|\mathcal{T}_{i,k} \right|} 
   & \leq C \sum_{k=i+2}^{K} h_{k}^{2} x_{k-1}^{\sigma-2} \int_{x_{k-1}}^{x_{k}} {\left(y-x_{i} \right)^{-\alpha}} \, dy \\
     & \leq C \left(N^{-r}i^{r-1}\right)^{2} \left(N^{-r}i^{r}\right)^{\sigma-2}  \int_{x_{i+1}}^{x_{K}} {\left(y-x_{i} \right)^{-\alpha}} \, dy \\
     & \leq C \left(N^{-r}i^{r-1}\right)^{2} \left(N^{-r}i^{r}\right)^{\sigma-2} \left(x_{K}-x_{i}\right)^{1-\alpha} \\
     & \leq C N^{-r \left( 1+\sigma -\alpha\right)}   i^ {r \left( 1+\sigma -\alpha\right)-2}.
\end{split}
\end{equation}
For the special case $K=2i <N$, from \eqref{eqn4.2}, \eqref{eqn3.1} and \eqref{eqn3.2}, it yields  
\begin{equation}\label{eqn4.10}
    \begin{split}
        \sum_{k=K+1}^{N}\left|\mathcal{T}_{i,k} \right|
     &\leq C \sum_{k=K+1}^{N} {h_{k}^{3}} x_{k}^{\sigma-2}\left(x_{k-1}-x_{i} \right)^{-\alpha} \\
     & \leq C \sum_{k=K+1}^{N} \left( N^{-r} k^{r-1}\right)^{3} \left(N^{-r} k^{r}\right)^{\sigma-2} \left(N^{-r} i^{r} \right)^{-\alpha}\\
     &\leq  C N^{-r\left(1+\sigma-\alpha\right)} i^{-r\alpha} \sum_{k=K+1}^{N}  k^{r\left(1+\sigma\right)-3}\\
     &\le \begin{cases}
C  N^{-r\left(1+\sigma -\alpha \right)} i^{-r\alpha},        &{\rm if }  ~~r \left(1+\sigma\right)<2, \\
C  N^{-r\left(1+\sigma -\alpha \right)} i^{-r\alpha}  \ln N, &{\rm if }  ~~r \left(1+\sigma\right)=2, \\                    
C  N^{-\left(2-r\alpha\right)}          i^{-r\alpha},        &{\rm if }  ~~r \left(1+\sigma\right)>2.                    
	\end{cases}
    \end{split}
\end{equation}
Adding \eqref{eqn4.8}, \eqref{eqn4.9} and \eqref{eqn4.10} gives the desired result.
\end{proof}

\begin{lemma}\label{nnlem4.3}
	Let $r>0$  and $0<\sigma <1$. Then  there exists a constant $C$ such that
	\[
	\sum_{k=N+1}^{2N} |\mathcal{T}_{i,k}| \le
\begin{cases}
C  N^{-r\left(1+\sigma\right)},          &{\rm if }   ~~r\left(1+\sigma\right)<2, \\
C  N^{-r\left(1+\sigma\right)} \ln N,    &{\rm if }   ~~r\left(1+\sigma\right)=2, \\
C  N^{-2},                               &{\rm if }   ~~r\left(1+\sigma\right)>2
	\end{cases}
	\]
	for all $i\in\{1, \dots, N\}$.
\end{lemma}
\begin{proof}

From \eqref{eqn4.2} and \eqref{eqn3.2}, we have
\begin{equation}\label{eqn4.11}
\begin{split}
    \left| \mathcal{T}_{i,N+1} \right| 
    & \leq C h_{N+1}^{2} \left( 2T- x_{N+1}\right)^{\sigma-2} \int_{x_{N}}^{x_{N+1}}\left(y -x_{i} \right)^{-\alpha} \, dy \\
     & \leq C h_{N+1}^{3-\alpha}= CN^{-\left( 3-\alpha\right) },
\end{split}
\end{equation}
since 
\[\int_{x_{N}}^{x_{N+1}} \left(y -x_{i} \right)^{-\alpha}\,dy  = \frac{1}{1-\alpha} \,h_{N+1}^{1-\alpha}~~{\rm for}~~i=N,\]
  and 
  \[\int_{x_{N}}^{x_{N+1}} \left(y -x_{i} \right)^{-\alpha}\,dy  
  \leq h_{N+1}\left(x_{N} -x_{i} \right)^{-\alpha} 
  \leq  h_{N+1}^{1-\alpha}~~{\rm for}~~i<N. \]

Furthermore, we can derive that, from \eqref{eqn4.2}, \eqref{eqn3.1} and \eqref{eqn3.2},
\begin{equation}\label{eqn4.12}
  \begin{split}
     \sum_{k=N+2}^{\left\lceil 3N/2 \right\rceil}    \left|\mathcal{T}_{i,k} \right|
      & \leq C  \sum_{k=N+2}^{\left\lceil 3N/2 \right\rceil} h_{k}^{2} \left(2T-x_{k} \right)^{\sigma-2} 
      \int_{x_{k-1}}^{x_{k}} \left(y-x_{i} \right)^{-\alpha} \,dy \\
       & \leq C N^{-2} \int_{x_{N+1}}^{x_{\left\lceil 3N/2 \right\rceil}} \left(y-x_{i} \right)^{-\alpha} \, dy \\
       & \leq C N^{-2}.
  \end{split}
\end{equation}

According to   \eqref{eqn3.1}, \eqref{eqn3.2} and \eqref{eqn4.2}, there exists
\begin{align}\label{eqn4.13}
\sum_{k=\left\lceil 3N/2 \right\rceil+1}^{2N-1}\left|\mathcal{T}_{i,k} \right| 
& \leq C  \sum_{k=\left\lceil 3N/2 \right\rceil+1}^{2N-1} h_{k}^{2} \left(2 T -x_{k} \right)^{\sigma-2}  \int_{x_{k-1}}^{x_{k}} \left(y-x_{i} \right)^{-\alpha} \,dy  \notag\\
       & \leq C \sum_{k=\left\lceil 3N/2 \right\rceil+1}^{2N-1} h_{k}^{3} \left( 2 T- x_{k}\right)^{\sigma-2} \notag\\
       & \leq C \sum_{k=\left\lceil 3N/2 \right\rceil+1}^{2N-1} \left(N^{-r} \left(2N+1-k\right)^{r-1} \right)^{3} \left( N^{-r} \left(2N-k\right)^{r} \right)^{\sigma-2}\notag \\
      & \leq C N^{-r\left(1+\sigma\right)}  \sum_{q=2}^{\left\lceil N/2 \right\rceil }  q^{r\left(1+\sigma\right)-3} \notag \\
       & \le \begin{cases}
C  N^{-r\left(1+\sigma\right)},          &{\rm if}~~  r\left(1+\sigma\right)<2, \\
C  N^{-r\left(1+\sigma\right)} \ln N,    &{\rm if}~~  r\left(1+\sigma\right)=2, \\
C  N^{-2},                               &{\rm if}~~  r\left(1+\sigma\right)>2.
	\end{cases}
\end{align}
By \eqref{eqn4.2} and \eqref{eqn3.2}, it yields 
\begin{equation}\label{eqn4.14}
\begin{split}
   \left|\mathcal{T}_{i,2N} \right| &\leq \int_{x_{2N-1}}^{x_{2N}}\frac{\left|u \left(y \right)- \Pi_h u \left( y\right) \right|}{\left( y-x_{i}\right)^{\alpha}} \, dy   \\
     & \leq C \int_{x_{2N-1}}^{x_{2N}} \left(2T-y \right)^{\sigma} \, dy
     + \frac{\left| u\left(x_{2N-1} \right)\right|}{h_{2N}} \int_{x_{2N-1}}^{x_{2N}} {\left(x_{2N}-y\right) \, dy} \\
     & \leq C h_{2N}^{1+\sigma} \leq C N^{-r\left(1+\sigma\right)}.
\end{split}
\end{equation}
 Adding \eqref{eqn4.11}, \eqref{eqn4.12}, \eqref{eqn4.13} and  \eqref{eqn4.14} gives the desired result.
\end{proof}

\begin{lemma}\label{nnlem4.4}
Let $r>0$  and $0<\sigma <1$. Then  there exists a constant $C$ such that
\begin{equation*}
  |R_{i}| 
  \leq  \begin{cases}
C  N^{-r \left(1+\sigma -\alpha \right)} i^{-r\alpha},        &{\rm if}~~   r \left(1+\sigma\right) <2, \\
C  N^{-r \left(1+\sigma -\alpha \right)} i^{-r\alpha}  \ln N, &{\rm if}~~   r \left(1+\sigma\right) =2, \\
C  N^{-\left(2-r\alpha\right)}           i^{-r\alpha},        &{\rm if}~~   r \left(1+\sigma\right) >2
	\end{cases}
\end{equation*}
for $i\in{1,2,\cdots,N}$, 
and
\begin{equation*}
  |R_{i}| 
  \leq  \begin{cases}
C  N^{-r \left(1+\sigma -\alpha \right)} \left(2N-i\right)^{-r\alpha},        &{\rm if}~~   r \left(1+\sigma\right) <2, \\
C  N^{-r \left(1+\sigma -\alpha \right)} \left(2N-i\right)^{-r\alpha}  \ln N, &{\rm if}~~   r \left(1+\sigma\right) =2, \\
C  N^{-\left(2-r\alpha\right)}           \left(2N-i\right)^{-r\alpha},        &{\rm if}~~   r \left(1+\sigma\right) >2
	\end{cases}
\end{equation*} 
for $i\in{N+1,N+2,\cdots,2N}$.
\end{lemma}

\begin{proof}
For $i=1,2\dots, N$, this result is an immediate consequence by Lemmas~\ref{nnlem4.1}--\ref{nnlem4.3}.
The similar arguments can be performed as Lemmas~\ref{nnlem4.1}--\ref{nnlem4.3} for $i\in{N+1,N+2,\cdots,2N}$,  we omit it here.
\end{proof}

\subsection{Convergence analysis}
We start with the global error analysis for the  steady-state problem \eqref{equ1.2} under the low regularity solution.
\begin{theorem}\label{nnthm4.5}
Let $r>0$  and $0<\sigma <1$. Then  there exist  constants $C$ such that
  \begin{equation*}
    \max_{1\leq i \leq 2N-1} \left|u\left(x_{i}\right)-u_{i}\right| 
\leq  \begin{cases}
C  N^{r-2} ,              &{\rm if}~~  r\left( 1+\sigma\right) >2,\\
C  N^{r-2} \ln N,         &{\rm if}~~  r\left( 1+\sigma\right) =2, \\
C  N^{-r \sigma},         &{\rm if}~~  r\left( 1+\sigma\right) <2. \\
	\end{cases}
  \end{equation*}
\end{theorem}

\begin{proof}
From  \eqref{nequ2.9},  \eqref{eqn3.6} and Lemma \ref{nnlem4.4}, we obtain,
 for $1\leq i_{0}\leq N$,
  \begin{equation*}
  \begin{split}
      \left\|e\right\|_{\infty} 
        \leq C\frac{\left|R_{i_{0}}\right|}{h_1 \left(x_{i_{0}}-x_{0}\right)^{-\alpha}}
        \leq 
        \begin{cases}
C  N^{r-2} ,              &{\rm if}~~     r\left( 1+\sigma\right) >2,\\
C  N^{r-2} \ln N,         &{\rm if}~~     r\left( 1+\sigma\right) =2, \\
C  N^{-r \sigma},         &{\rm if}~~     r\left( 1+\sigma\right) <2, \\
	\end{cases}
  \end{split}
      \end{equation*}
   and,   for $N+1\leq i_{0} \leq 2N-1$, 
  \begin{equation*}
  \begin{split}
      \left\|e\right\|_{\infty} 
      \leq C\frac{\left|R_{i_{0}}\right|}{h_{2N} \left(x_{2N}-x_{i_{0}}\right)^{-\alpha}}
        \leq 
        \begin{cases}
C  N^{r-2} ,             &{\rm if}~~     r\left(1+\sigma\right) >2,\\
C  N^{r-2} \ln N,        &{\rm if}~~     r\left(1+\sigma\right) =2, \\
C  N^{-r \sigma},        &{\rm if}~~     r\left(1+\sigma\right) <2. \\
	\end{cases}
  \end{split}
      \end{equation*}
    The proof is completed. 
\end{proof}

\section{Error analysis: Time-dependent case with low regularity solution}
An optimal error estimate with far less than first-order accuracy is demonstrated in section 4
for the steady-state problems. Now, we further study time-dependent case, however, 
second-order convergence can be achieved on graded meshes. 

\subsection{Local truncation error}
Before we start to discuss the time-dependent nonlocal problem \eqref{equ1.1}
we shall briefly review and revise the local truncation error for the corresponding stationary problem \eqref{equ1.2}.
\begin{lemma}\label{nnlem5.1}
Let $r>0$  and $0<\sigma <1$.   Then  there exists a constant $C$ such that
    \begin{equation*}
        \left|R_{i}\right|  \leq \begin{cases}
C  N^{-r\left(1+\sigma -\alpha \right)},      &{\rm if }  ~~r \left(1+\sigma- \alpha\right)<2, \\
C  N^{-r\left(1+\sigma -\alpha \right)}\ln N, &{\rm if }  ~~r \left(1+\sigma- \alpha\right)=2, \\                    
C  N^{-2},                                    &{\rm if }  ~~r \left(1+\sigma- \alpha\right)>2                
	\end{cases}
    \end{equation*}
    for $i=1,2,\cdots,2N-1$.
\end{lemma}

\begin{proof}
We only consider $1\leq i\leq N$; the case of $N+1 \leq i \leq 2N-1$ can be proved  in the same way.

   From Lemma \ref{nnlem4.1}, it is clear that
    \begin{equation}\label{eqn5.1}
        \sum_{k=1}^{i} \left|\mathcal{T}_{i,k}\right| \leq
        \begin{cases}
           C N^{-r\left(1+\sigma-\alpha\right)},  &{\rm if}~~ r\left(1+\sigma-\alpha\right) \leq 2,\\
           C N^{-2},                              &{\rm if}~~ r\left(1+\sigma-\alpha\right) > 2.
        \end{cases}
    \end{equation}

Taking $K=\min \left\{2i,N\right\}$, from \eqref{eqn4.8} and \eqref{eqn4.9}, we can get
\begin{equation*}
    \sum_{k=i+1}^{K} \left|\mathcal{T}_{i,k}\right|
    \leq \begin{cases}
           C N^{-r\left(1+\sigma-\alpha\right)},  &{\rm if}~~ r\left(1+\sigma-\alpha\right) \leq 2,\\
           C N^{-2},                              &{\rm if}~~ r\left(1+\sigma-\alpha\right) > 2.
        \end{cases}
\end{equation*}
In particular, for $K=2i <N$,  by \eqref{eqn4.10}, it yields 
    \begin{equation*}
    \begin{split}
        \sum_{k=K+1}^{N}\left|\mathcal{T}_{i,k} \right|
     & \leq C \sum_{k=K+1}^{N} \left( N^{-r} k^{r-1}\right)^{3} \left(N^{-r} k^{r}\right)^{\sigma-2} \left(N^{-r} k^{r} \right)^{-\alpha}\\
     &\leq  C N^{-r\left(1+\sigma-\alpha\right)}  \sum_{k=K+1}^{N}  k^{r\left(1+\sigma- \alpha\right)-3}\\
     &\le \begin{cases}
C  N^{-r\left(1+\sigma -\alpha \right)} ,     &{\rm if }  ~~r \left(1+\sigma- \alpha\right)<2, \\
C  N^{-r\left(1+\sigma -\alpha \right)}\ln N, &{\rm if }  ~~r \left(1+\sigma- \alpha\right)=2, \\                    
C  N^{-2},                                    &{\rm if }  ~~r \left(1+\sigma- \alpha\right)>2.                    
	\end{cases}
    \end{split}
\end{equation*}
Then we can conclude the following results 
\begin{equation}\label{equa5.3}
\sum_{k=i+1}^{N} \left|\mathcal{T}_{i,k} \right| 
\leq \begin{cases}
C  N^{-r\left(1+\sigma -\alpha \right)},      &{\rm if }  ~~r \left(1+\sigma- \alpha\right)<2, \\
C  N^{-r\left(1+\sigma -\alpha \right)}\ln N, &{\rm if }  ~~r \left(1+\sigma- \alpha\right)=2, \\                    
C  N^{-2},                                    &{\rm if }  ~~r \left(1+\sigma- \alpha\right)>2.                    
	\end{cases}
\end{equation}

From Lemma \ref{nnlem4.3}, we can check that 
\begin{equation}\label{eqn5.2}
    \sum_{k=N+1}^{2N} \left|\mathcal{T}_{i,k}\right|     
  \leq \begin{cases}
           C N^{-r\left(1+\sigma-\alpha\right)},  &{\rm if}~~ r\left(1+\sigma-\alpha\right) \leq 2,\\
           C N^{-2},                              &{\rm if}~~ r\left(1+\sigma-\alpha\right) > 2,
        \end{cases}
\end{equation}
since, for $r\left(1+\sigma\right) \leq 2$, there exist 
$N^{-r\left(1+\sigma\right)}\leq N^{-r\left(1+\sigma\right)}\ln N \leq C N^{-r\left(1+\sigma-\alpha\right)}$ 
and $r\left(1+\sigma-\alpha\right) < 2$;
moreover, $N^{-2}\leq N^{-r\left(1+\sigma-\alpha\right)}$ for $r\left(1+\sigma-\alpha\right) \leq 2$.

From \eqref{eqn5.1}, \eqref{equa5.3} and \eqref{eqn5.2}, we  obtain the desired results. 
\end{proof}

\subsection{Convergence and stability  analysis}
We mainly focus on   the convergence analysis for the time-dependent nonlocal problem \eqref{equ1.1}, since the stability  can be easily obtained by Lax's equivalence theorem \cite[p.\,189]{Lev:07}. 

\begin{theorem}\label{nnthm5.2}
Let $u_{i}^{k}$ be the approximate solution of $u(x_{i},t_{k})$ computed by the discretization scheme \eqref{nequ2.6} . Then
\[
\max_{\substack{1\leq i \leq 2N-1\\1\leq k \leq M}} |u\left(x_{i},t_{k}\right)-u_{i}^{k}| 
\leq \begin{cases}
C  \left(N^{-r\left(1+\sigma -\alpha \right)} +M^{-2}\right),     &{\rm if }  ~~r \left(1+\sigma- \alpha\right)<2, \\
C  \left(N^{-r\left(1+\sigma -\alpha \right)}\ln N+M^{-2}\right), &{\rm if }  ~~r \left(1+\sigma- \alpha\right)=2, \\                 
C  \left(N^{-2}+M^{-2}\right),                                    &{\rm if }  ~~r \left(1+\sigma- \alpha\right)>2.                    
	\end{cases}
\]
\end{theorem}

\begin{proof}
Let $e_{i}^{k}=u\left(x_{i},t_{k}\right)-u_{i}^{k}$, with $e_{i}^{0}=0$, $i=1,2,\cdots,2N-1$, $k=0,1,\cdots,M$ and $E^{k}=\left(e_{1}^{k},e_{2}^{k}, \cdots,e_{2N-1}^{k} \right)^T$.  From  \eqref{nequ2.6} with  perturbation equation, we obtain
  \begin{equation}\label{eq7}
    \left(1+\frac{\tau}{2}a_{i,i} \right)e_{i}^{k} = e_{i}^{k-1} -\frac{\tau}{2} \sum_{j=1}^{2N-1} a_{i,j}e_{j}^{k-1} -\frac{\tau}{2}\sum_{j=1,j\ne i}^{2N-1}a_{i,j}e_{j}^{k}
    + \tau R_{i}^{k-\frac{1}{2}},
  \end{equation}
where the local truncation error is  $\left|R_{i}^{k-\frac{1}{2}}\right| \leq C\left (|R_i|+M^{-2}\right)$ and $R_i$ is given in Lemma \ref{nnlem5.1}.

Let $\left|e_{i_{0}}^{k} \right|:=\|E^k\|_{\infty} =\max\limits_{j=1,2,\cdots,2N-1}\left|e_{j}^{k} \right|$.  From \eqref{eq7} and Lemma \ref{nnlem2.4}, we have 
  \begin{equation*}
    \begin{split}
       \left(1+\frac{\tau}{2}a_{i_{0},i_{0}} \right)\left|e_{i_{0}}^{k}\right| 
       & \le \left|e_{i_{0}}^{k-1}\right| +\frac{\tau}{2 }\sum_{j=1}^{2N-1}\!\left|a_{i_{0},j}\right|\left|e_{j}^{k-1}\right| 
       +\frac{\tau}{2}\sum_{j=1,j\ne i_{0}}^{2N-1} \left|a_{i_{0},j}\right| \left|e_{j}^{k} \right|
       +\tau\left|R_{i_{0}}^{k-\frac{1}{2}}\right| \\
         & \le \left(1+\tau a_{i_{0},i_{0}} \right) \|E^{k-1}\|_{\infty} + \frac{\tau}{2}a_{i_{0},i_{0}} \left|e_{i_{0}}^{k}\right| +\tau\left|R_{i_{0}}^{k-\frac{1}{2}}\right|. \\
    \end{split}
  \end{equation*}

Let  $\|R^{k-\frac{1}{2}}\|_{\infty} =\max\limits_{j=1,2,\cdots,2N-1}\left|R_{j}^{k-\frac{1}{2}} \right|$, for $k=1,2,\cdots,M$. Form \eqref{nequ2.8}, it yields
\[a_{i_{0},i_{0}}<\frac{2^{2-\alpha}}{1-\alpha}T^{1-\alpha}:=C_a.\]
Thus, using  the above equations with $\|E^{0}\|_{\infty}=0$, we get that
  \begin{equation*}
    \begin{split}
       \|E^{k}\|_{\infty} 
       & \leq \left(1+\tau C_a\right)\|E^{k-1}\|_{\infty}
       +\tau \|R^{k-\frac{1}{2}}\|_{\infty} \\
         & \leq \left(1+\tau C_{a}\right)^{k}\|E^{0}\|_{\infty} 
         +\tau  \sum_{l=1}^{k}\left(1+\tau C_a\right)^{k-l} \|R^{l-\frac{1}{2}}\|_{\infty}\\
         & \leq \begin{cases}
C \left(N^{-r\left(1+\sigma -\alpha \right)} +M^{-2}\right),     &{\rm if }  ~~r \left(1+\sigma- \alpha\right)<2, \\
C \left(N^{-r\left(1+\sigma -\alpha \right)}\ln N+M^{-2}\right), &{\rm if }  ~~r \left(1+\sigma- \alpha\right)=2, \\                 
C \left(N^{-2}+M^{-2}\right),                                    &{\rm if }  ~~r \left(1+\sigma- \alpha\right)>2.                 
	\end{cases}\\
    \end{split}
  \end{equation*}
  The proof is completed.
\end{proof}

\begin{remark}\label{rem5.1}
The above error analysis are easily extended  for  certain multidimensional problems, e.g., 
    \begin{equation}\label{eqn5.3}
     u_t\left(x,y,t\right)- \int_{\Omega} \frac{u(\bar{x},\bar{y},t)-u(x,y,t)}{\left|x-\bar{x}\right|^{\alpha}\left|y-\bar{y}\right|^{\beta}} \,d\bar{x}\,d\bar{y}
     = f\left(x,y,t\right).
    \end{equation}
 From  \cite{Cao:20} and \eqref{nequ2.6}, the full discretization of \eqref{eqn5.3} has the matrix form
    \begin{equation}\label{neqn5.5}
\left(I+\frac{\tau }{2}\mathcal{A}\right) U^{k}=\left(I-\frac{\tau }{2}\mathcal{A}\right) U^{k-1}+ \tau F^{k-\frac{1}{2}},~~~k=1,2,\cdots,M. 
\end{equation}
Here $   
 \mathcal{A}:=D_{\alpha}\otimes D_{\beta}-G_{\alpha}\otimes G_{\beta}
   $
and $D_{\alpha}$,  $D_{\beta}$, $G_{\alpha}$, $G_{\beta}$  are defined in \eqref{nequ2.7}. Moreover, the grid function are defined by
\begin{equation*}
  \begin{split}
    & U^{k}=\left(U_{1}^{k},U_{2}^{k},\cdots,U_{2N-1}^{k}\right)^{T}, ~U_{i}^{k}=\left(u_{i,1}^{k},u_{i,2}^{k},\cdots,u_{i,2N-1}^{k}\right).\\
  \end{split}
\end{equation*}

\end{remark}

\section{Numerical experiments}\label{Se:NE}
We numerically verify the above theoretical results including the   convergence rates.  The maximum norm is employed to measure the numerical errors.
All  numerical experiments are executed in  Julia 1.10.0.

\begin{example}
{\bf Steady-state problems with smooth solution.}    
Consider the steady-state problems \eqref{equ1.2}   in the domain $0<x<1$.
The exact solution of the equation is $u(x)=e^{x}\sin{x}$.
Here the forcing function is calculated by the JacobiGL algorithm, as demonstrated in  \cite{CD:15} or \cite[p.\,448]{HW:08}.
\end{example}

\begin{table}[!ht]
{\footnotesize 
\begin{center}
\caption{Maximum errors and convergence rates of \eqref{nequ2.5} with smooth solution}
{\begin{tabular}{|c| r r r |r r r|} \hline
         \multirow{2}{*}{$r$}  &  \multicolumn{3}{c|}{$\alpha=0.3$}  &  \multicolumn{3}{c|}{$\alpha=0.7$}     \\ \cline{2-7}
                              &  ${ N=25}$    & ${N=50}$     & ${ N=100}$     &  ${ N=25}$    & ${ N=50}$     & ${N=100}$     \\ \hline
       \multirow{2}{*}{$\frac{1}{10}$} &2.7121E-01   & 2.3141E-01    & 1.9727E-01    &3.2843E-01    & 2.8221E-01    & 2.4216E-01         \\
                               &     &0.2290           & 0.2303                &     &0.2188           & 0.2208                  \\  \hline
         \multirow{2}{*}{$\frac{2}{3}$}  &1.1367E-02  & 4.7263E-03    & 1.9630E-03    &1.6104E-02    & 6.6190E-03    & 2.7124E-03         \\
                               &      &1.2660          &1.2676                & &1.2827                & 1.2870                   \\ \hline                 
         \multirow{2}{*}{1} &2.3460E-02   & 1.1669E-02    & 5.8160E-03   &3.2383E-02   & 1.5898E-02    & 7.8451E-03        \\
                               &      &1.0075          & 1.0046         & &1.0264                & 1.0189                  \\ \hline
\multirow{2}{*}{$\frac{3}{2}$}  &1.7322E-01  & 1.2272E-01    & 8.6908E-02    &2.1307E-01    & 1.5136E-01    & 1.0777E-01         \\
                               &      &0.4972           &0.4979                 & &0.4934                 & 0.4900                    \\ \hline  
         \multirow{2}{*}{4} &2.6637E+03   & 1.0809E+04    &4.3522E+04   &2.6455E+03   & 1.1593E+04   & 4.9022E+04       \\
                               &        &-2.0208        & -2.0095         & & -2.1316        & -2.0802            \\  \hline
\end{tabular}\label{smooth1}}
\end{center}}
\end{table}

Table \ref{smooth1}  shows  that the standard graded meshes are worse than the uniform grid and  even lead to divergence  for $r\geq 2$.
It also indicates that an optimal convergence rate (great than one) arises in anomalous graded meshes if $r=\frac{2}{3}$, which  are in agreement with Theorem \ref{nnthm3.5}.

\begin{example}
{\bf Steady-state problems with low regularity solution.} Consider the steady-state problems \eqref{equ1.2}  under  the  low regularity solution in the domain $0<x<1$.
The exact solution of the equation is $u(x)=e^{x}x^{0.3}\left(1-x\right)^{0.3}$.
Here the forcing function is calculated by the JacobiGL algorithm \cite{CD:15,HW:08}.
\end{example}
\begin{table}[!ht]
{\footnotesize \begin{center}
\caption{Maximum errors and convergence rates of \eqref{nequ2.5} with low regularity solution}
{\begin{tabular}{|c| r r r |r r r|} \hline
     \multirow{2}{*}{$r$}  &  \multicolumn{3}{c|}{$\alpha=0.3,\,\sigma=0.3$} &\multicolumn{3}{c|}{$\alpha=0.7,\,\sigma=0.3$} \\\cline{2-7}
                              &  ${ N=25}$    & ${N=50}$     & ${ N=100}$     &  ${ N=25}$    & ${ N=50}$     & ${N=100}$      \\ \hline
       \multirow{2}{*}{$\frac{1}{10}$}    & 7.9074E-01    & 7.7541E-01   & 7.5993E-01   & 8.3789E-01    & 8.2387E-01    & 8.0921E-01     \\
                               &             & 0.0282          &0.0291             &               & 0.0243          & 0.0259      \\  \hline
         \multirow{2}{*}{1}    & 3.4398E-01    & 2.7800E-01    & 2.2493E-01    & 3.6838E-01    & 2.9473E-01    & 2.3629E-01    \\
                               &               &0.3072          &0.3056         &               & 0.3218          &0.3188          \\ \hline
    \multirow{2}{*}{$\frac{2}{1+\sigma}$}    & 3.4263E-01    & 2.7098E-01   & 2.1273E-01  & 3.7335E-01    & 2.9526E-01    & 2.3190E-01    \\
                               &               & 0.3385          & 0.3492          &               & 0.3385          & 0.3485        \\ \hline
         \multirow{2}{*}{4}    & 7.8187E+02    & 3.1534E+03   & 1.2657E+04  & 7.8879E+02   & 3.4092E+03   & 1.4311E+04     \\                            
         &               & -2.0119          & -2.0050          &               & -2.1117        & -2.0697         \\  \hline
\end{tabular}\label{ODElow:1}}
\end{center}}
\end{table}

Table \ref{ODElow:1}  shows  that it is contrary to the results of smooth functions, its sharp error estimate (far less than one) appears in standard graded meshes with the low regularity of solution \eqref{equ1.4}, 
 which  are in agreement with  Theorem \ref{nnthm4.5}.

 \begin{example}
{\bf Time-dependent problems for 1D and 2D.}
Consider the time-dependent  problems \eqref{equ1.1} for 1D and \eqref{rem5.1} for 2D  under  the  low regularity solution in the domain $0<x,y<1$, $0<t<1$.
The exact solution of the equations are  $u(x,t)=e^{x+t}x^{0.3}\left(1-x\right)^{0.3}$
and 
\[u(x,y,t)=e^{x+y+t}x^{0.3}\left(1-x\right)^{0.3}y^{0.3}\left(1-y\right)^{0.3},\] 
respectively. 
Here the forcing function is calculated by the JacobiGL algorithm \cite{CD:15,HW:08}.
\end{example}

\begin{table}[!ht]
{\footnotesize \begin{center}
\caption{Maximum errors and convergence rates of \eqref{nequ2.6} with low regularity solution}
{\begin{tabular}{|c| r r r |r r r|} \hline
\multirow{2}{*}{$r$}  &  \multicolumn{3}{c|}{$\alpha=0.3,\,\sigma=0.3$}  &  \multicolumn{3}{c|}{$\alpha=0.7,\,\sigma=0.3$}     \\ \cline{2-7}
                              &  ${ N=25}$    & ${N=50}$     & ${ N=100}$     &  ${ N=25}$    & ${ N=50}$     & ${N=100}$       \\ \hline
\multirow{2}{*}{$\frac{1}{10}$}    & 5.7985E-01   & 5.4054E-01   & 5.0377E-01    & 9.1193E-01    & 8.6425E-01    & 8.1921E-01     \\
                               &               & 0.1013          & 0.1016        &               & 0.0775          & 0.0772           \\  \hline
\multirow{2}{*}{1}    & 2.7978E-02   & 1.3211E-02   & 6.2824E-03    & 1.0937E-01    & 6.6510E-02    & 4.1126E-02     \\
                               &               & 1.0825          &1.0724        &               & 0.7175         & 0.6935          \\  \hline
\multirow{2}{*}{$\frac{2}{1+\sigma-\alpha}$}   & 2.5376E-03   & 6.7806E-04   & 1.8116E-04   & 8.6346E-03   & 2.3398E-03    & 5.9992E-04     \\
                               &               &1.9040         & 1.9041         &               & 1.8838         &1.9635          \\ \hline
\multirow{2}{*}{4}  & 2.9415E-03    & 7.4086E-04   & 1.8591E-04   & 8.5118E-03   & 2.2250E-03    & 5.7114E-04    \\
                               &               & 1.9893           & 1.9946           &               & 1.9357           & 1.9619       \\ \hline
\end{tabular}\label{PDE1}}
\end{center}}
\end{table}

\begin{table}[!ht]
{\footnotesize \begin{center}
\caption{Maximum errors and convergence rates of \eqref{neqn5.5} with low regularity solution}
{\begin{tabular}{|c| r r r |r r r|} \hline
\multirow{2}{*}{$r$}  &  \multicolumn{3}{c|}{$\alpha=0.3,\,\sigma=0.3$}  &  \multicolumn{3}{c|}{$\alpha=0.7,\,\sigma=0.3$}     \\ \cline{2-7}
                              &  ${ N=10}$    & ${N=20}$     & ${ N=40}$     &  ${ N=10}$    & ${ N=20}$     & ${N=40}$       \\ \hline
\multirow{2}{*}{1}    & 2.1188E-01   & 9.8586E-02   & 4.5734E-02   & 8.9027E-01    & 5.1261E-01    & 2.8362E-01     \\
                               &               & 1.1038           &1.1081         &               & 0.7964          & 0.8539           \\  \hline
\multirow{2}{*}{$\frac{2}{1+\sigma-\alpha}$}   & 4.0795E-02   & 1.0899E-02   & 2.9089E-03   & 3.1494E-01   & 8.5428E-02    & 2.2781E-02    \\
                               &               &1.9042          & 1.9057          &               & 1.8823          & 1.9069          \\ \hline
\multirow{2}{*}{4}  & 5.7365E-02    & 1.4645E-02  & 3.6988E-03   & 3.8907E-01   & 1.0463E-01    & 2.7542E-02    \\
                               &               & 1.9697            & 1.9853            &               & 1.8948            & 1.9255        \\ \hline
\end{tabular}\label{PDE2}}
\end{center}}
\end{table}

Table \ref{PDE1} and  \ref{PDE2} show  that second-order convergence $\mathcal{O}\left( N^{-\min \left\{r \left(1+\sigma-\alpha\right), \,2 \right\} } \right)$ can be achieved on graded meshes
for the time-dependent nonlocal problems \eqref{equ1.1} including two-dimensional case under the low regularity of solution \eqref{equ1.4}, as demonstrated in Theorem \ref{nnthm5.2} and Remark \ref{rem5.1}.
\section{Conclusions}
The nonlocal diffusion  models  involve the weakly singular kernels, which  exhibit  a severe order reduction by  many   methods.
In this work we first derive an  optimal error estimate (far less than second-order)  for the steady-state counterpart either the standard  graded meshes or anomalous graded meshes,  under  the sufficient smooth and mild  regularity   of the solution. And the second-order convergence rate are well established for the time-dependent problems including two-dimensional case  on graded meshes.
Based on the idea of  Lemmas \ref{nnlem4.4} and \ref{nnlem5.1}, 
it is  interesting to further analyze the second-order convergence rate  for  Riesz fractional  operator \cite{CD:14,HW:22,Miller:93}
\begin{equation*}
\frac{\partial ^{\alpha+1}}{\partial |x|^{\alpha+1}}u(x) =\kappa_{\alpha}\frac{\partial^2 u}{\partial x^2} \int^b_a \frac{u(y)}{|x-y|^\alpha}dy,~0<\alpha<1
\end{equation*}
with a   constant  $\kappa_{\alpha}$.



\end{document}